\documentclass[a4paper,12pt]{article}
 
\usepackage{latexsym,amssymb,amsmath, amsthm}  

\newtheorem{thm}{Theorem}
\newtheorem{lem}{Lemma}

\newtheorem{rem}{Remark}

\begin{document}

\title{Biaxial monogenic functions from Funk-Hecke's formula combined with Fueter's theorem}

\author{Dixan Pe\~na Pe\~na\\\small{e-mail: dpp@cage.ugent.be} \and Frank Sommen\\\small{e-mail: fs@cage.ugent.be}}

\date{\normalsize{Clifford Research Group, Department of Mathematical Analysis\\Faculty of Engineering and Architecture\\Ghent University\\Galglaan 2, 9000 Gent, Belgium}}

\maketitle

\begin{abstract}
\noindent  Funk-Hecke's formula allows a passage from plane waves to radially invariant functions. It may be adapted to transform axial monogenics into biaxial monogenics that are monogenic functions invariant under the product group SO($p$)$\times$SO($q$). Fueter's theorem transforms holomorphic functions in the plane into axial monogenics, so that by combining both results, we obtain a method to construct biaxial monogenics from  holomorphic functions.\vspace{0.2cm}\vspace{0.1cm}
\end{abstract}

\section{Introduction}

Clifford analysis (see e.g. \cite{BDS,DSS,GuSp}) forms a unifying language for all kinds of higher dimensional generalizations of the Cauchy-Riemann system; in this way it also constitutes a natural framework within which the equations of mathematical physics may be elegantly formulated. The basic language is defined as follows.

The real Clifford algebra $\mathbb R_{0,m}$ (see \cite{Cl}) is the free algebra generated by the symbols $e_1,\dots,e_m$, subject to the multiplication relations 
\[e_je_k+e_ke_j=-2\delta_{jk},\quad j,k=1,\dots,m.\] 
The dimension of the real Clifford algebra $\mathbb R_{0,m}$ is $2^m$, as is the case for the Grassmann algebra generated by $e_1,\dots,e_m$, but the difference is that now $e_j^2=-1$ instead of $e_j^2=0$, creating a structure with similarities to the complex numbers. 

The elements $e_j$ are geometrically identified with the Euclidean basis of the vector space $\mathbb R^m$, and the Clifford (or geometric) product between two vectors $\underline X=\sum_{j=1}^mx_je_j$ and $\underline Y=\sum_{j=1}^my_je_j$ decomposes into the scalar valued dot product $\underline X \bullet\underline Y$, which coincides, up to a minus sign, with the standard Euclidean inner product, and the bivector valued wedge product $\underline X \wedge\underline Y$. A general element $a$ of $\mathbb R_{0,m}$ may be written as
\[a=\sum_Aa_Ae_A,\quad a_A\in\mathbb R,\]
in terms of the basis elements $e_A=e_{j_1}\dots e_{j_k}$, defined for every subset $A=\{j_1,\dots,j_k\}$ of $\{1,\dots,m\}$ with $j_1<\dots<j_k$. For the empty set, one puts $e_{\emptyset}=1$, the latter being the identity element. The subspace
\[\mathbb R_{0,m}^+=\left\{a\in\mathbb R_{0,m}:\;a=\sum_{\vert A\vert\;\textrm{even}}a_Ae_A\right\}\]
is called the even subalgebra.

The first order differential operator in $\mathbb R^m$ given by 
\[\partial_{\underline X}=\sum_{j=1}^me_j\partial_{x_j}\]
is called the Dirac operator, and its null solutions are called (left) monogenic functions, i.e. solutions of $\partial_{\underline X}f(\underline X)=0$ (see \cite{BDS,DSS}). In a similar way one also defines monogenicity with respect to the generalized Cauchy-Riemann operator $\partial_{x_0}+\partial_{\underline X}$ in $\mathbb R^{m+1}$. It is well-known that every monogenic function is also harmonic, which is a consequence of the following fact
\[\Delta=\partial_{x_0}^2+\Delta_{\underline X}=\sum_{j=0}^m\partial_{X_j}^2=(\partial_{x_0}+\partial_{\underline X})(\partial_{x_0}-\partial_{\underline X}).\]
One basic result in Clifford analysis is the so-called Fueter's theorem, which discloses a remarkable connection existing between holomorphic functions and monogenic functions (see \cite{F,Q,Sce,S3}). For other generalizations and works about Fueter's theorem we refer the reader e.g. to \cite{CSDPF,CoSaF,CoSaF2,KQS,DS2,QS}. 

\begin{thm}[Fueter's theorem]
Let $u+iv$ be a holomorphic function in the open subset $\Xi$ of the upper half-plane and assume that $P_k(\underline X)$ is a homogeneous monogenic polynomial of degree $k$ in $\mathbb R^m$. If $m$ is odd, then the function
\begin{equation*}
\Delta^{k+\frac{m-1}{2}}\left(\left(u(x_0,R)+\frac{\underline X}{R}\,v(x_0,R)\right)P_k(\underline X)\right),\quad R=\vert\underline X\vert
\end{equation*}
is monogenic in $\Omega=\{(x_0,\underline X)\in\mathbb R^{m+1}:\;(x_0,R)\in\Xi\}$.
\end{thm}

Fueter's theorem leads to the construction of special monogenic functions starting from the choice of the initial holomorphic function $u+iv$ (see \cite{HDPS,DS1}). The type of functions generated are of the form
\begin{equation*}\label{axialmonfun}
\left(A(x_0,R)+\frac{\underline X}{R}\,B(x_0,R)\right)P_k(\underline X),
\end{equation*}
where $A$ and $B$ are continuously differentiable $\mathbb R$-valued functions. These functions are called axial monogenic functions of degree $k$ (see \cite{DSS,LB,S1,S2}) and one can easily show that $A$ and $B$ must satisfy the following Vekua-type system 
\begin{equation*}
\begin{split}
\partial_{x_0}A-\partial_RB&=\frac{2k+m-1}{R}B\\
\partial_RA+\partial_{x_0}B&=0.
\end{split}
\end{equation*}
Notice that axial monogenics consist of a factor $A(x_0,R)+\frac{\underline X}{R}\,B(x_0,R)$ that is invariant under the group SO($m$) of rotations around the $x_0$-axis, multiplied with the factor $P_k(\underline X)$ that is \lq\lq spherical monogenics" of degree $k$ in $\mathbb R^m$. 

In Section \ref{sect2} of this paper we further show how axial monogenics may be embedded in the biaxial decomposition $\mathbb R^p\oplus\mathbb R^{q}$ of $\mathbb R^{m}$ and this by choosing the axis along a chosen unit vector $\underline t\in\mathbb R^p$, leading to axial monogenics depending on a parameter unit vector $\underline t$. By integrating the parameter vector $\underline t$ over the sphere $S^{p-1}$ and applying Funk-Hecke's formula one arrives at a transformation from $\underline t$-dependent axial monogenics into biaxial monogenics. This method of combining Funk-Hecke's formula with Fueter's theorem is elaborated in Section \ref{sect3}.    

\section{The operator $\langle\underline t,\partial_{\underline x}\rangle-\underline t\partial_{\underline y}$}\label{sect2}

Consider the biaxial splitting $\mathbb R^{m}=\mathbb R^p\oplus\mathbb R^{q}$, $m=p+q$. In this way, for any $\underline X\in\mathbb R^{m}$ we may write  
\[\underline X=\underline x+\underline y,\] 
where
\[\underline x=\sum_{j=1}^px_je_j\quad\text{and}\quad\underline y=\sum_{j=1}^qx_{p+j}e_{p+j}.\]
By the above, we can also split the Dirac operator $\partial_{\underline X}$ as 
\[\partial_{\underline X}=\partial_{\underline x}+\partial_{\underline y},\]
the operators $\partial_{\underline x}$ and $\partial_{\underline y}$ being given by 
\[\partial_{\underline x}=\sum_{j=1}^pe_j\partial_{x_j},\qquad\partial_{\underline y}=\sum_{j=1}^qe_{p+j}\partial_{x_{p+j}}.\]
Let $\underline t=\sum_{j=1}^pt_je_j$ be an arbitrary fixed unit vector in $\mathbb R^p$ and let us introduce the following differential operator 
\begin{equation}\label{opeFS}
\langle\underline t,\partial_{\underline x}\rangle-\underline t\partial_{\underline y},\qquad\langle\underline t,\partial_{\underline x}\rangle=\sum_{j=1}^pt_j\partial_{x_j}.
\end{equation}
Suppose  that $P_{\ell}(\underline y)$ is a given arbitrary homogeneous monogenic polynomial of degree $\ell$ in $\mathbb R^q$. In this section we shall look for special null solutions of the operator (\ref{opeFS}), which we assume to be of the form
\begin{equation}\label{axialwfunc}
\bigl(M(\theta,\rho)-\underline t\,\underline\nu\,N(\theta,\rho)\bigr)P_{\ell}(\underline y),\quad\underline\nu=\frac{\underline y}{\vert\underline y\vert},
\end{equation}
where $M$ and $N$ are continuously differentiable $\mathbb R$-valued functions depending on the two variables 
\[(\theta,\rho)=(\langle\underline x,\underline t\rangle,\vert\underline y\vert).\]

\begin{lem}\label{axialwavef}
Assume that $M$ and $N$ are continuously differentiable $\mathbb R$-valued functions of two variables defined on an open subset of the upper half-plane. The function  
\[\bigl(M(\theta,\rho)-\underline t\,\underline\nu\,N(\theta,\rho)\bigr)P_{\ell}(\underline y)\]
is a null solution of the operator $\langle\underline t,\partial_{\underline x}\rangle-\underline t\partial_{\underline y}$ if and only if 
\begin{equation}\label{VEopeFS}
\begin{split}
\partial_{\theta}M-\partial_{\rho}N&=\frac{2\ell+q-1}{\rho}N\\
\partial_{\rho}M+\partial_{\theta}N&=0.
\end{split}
\end{equation}
\end{lem}
\begin{proof}
It is easily seen that 
\begin{align*}
\langle\underline t,\partial_{\underline x}\rangle M&=\sum_{j=1}^pt_j\partial_{x_j}M=\sum_{j=1}^pt_j(\partial_{\theta}M)(\partial_{x_j}\theta)=\vert\underline t\vert^2\partial_{\theta}M=\partial_{\theta}M\\
\partial_{\underline y}M&=\sum_{j=1}^qe_{p+j}\partial_{x_{p+j}}M=\sum_{j=1}^qe_{p+j}(\partial_{\rho}M)(\partial_{x_{p+j}}\rho)=\underline\nu(\partial_{\rho}M).
\end{align*}
Therefore
\begin{align*}
\langle\underline t,\partial_{\underline x}\rangle\left(MP_{\ell}\right)&=(\partial_{\theta}M)P_{\ell}\\
\langle\underline t,\partial_{\underline x}\rangle\left(\underline t\,\underline\nu NP_{\ell}\right)&=\underline t\,\underline\nu(\partial_{\theta}N)P_{\ell}\\
\partial_{\underline y}\left(MP_{\ell}\right)&=\big(\partial_{\underline y}M\big)P_{\ell}+M\big(\partial_{\underline y}P_{\ell}\big)=\underline\nu(\partial_{\rho}M)P_{\ell}.
\end{align*}
Using the Leibniz rule 
\begin{equation*}\label{lr}
\partial_{\underline y}(\underline yf)=-qf-2\sum_{j=1}^qx_{p+j}(\partial_{x_{p+j}}f)-\underline y(\partial_{\underline y}f),
\end{equation*}
and Euler's theorem for homogeneous functions, we also obtain that
\begin{align*}
\partial_{\underline y}\left(\underline t\,\underline\nu NP_{\ell}\right)&=-\underline t\,\partial_{\underline y}\left(\underline\nu NP_{\ell}\right)\\
&=\underline t\left(\partial_{\rho}N+\frac{2\ell+q-1}{\rho}N\right)P_{\ell}.
\end{align*}
We thus get
\[(\langle\underline t,\partial_{\underline x}\rangle-\underline t\partial_{\underline y})\Bigl(\bigl(M(\theta,\rho)-\underline t\,\underline\nu\,N(\theta,\rho)\bigr)P_{\ell}\Bigr)\]
\[\left(\left(\partial_{\theta}M-\partial_{\rho}N-\frac{2\ell+q-1}{\rho}N\right)-\underline t\,\underline\nu(\partial_{\rho}M+\partial_{\theta}N)\right)P_{\ell},\]
and the lemma easily follows.
\end{proof}

\begin{rem}\label{obser1}
Let $\left\{\underline t,\underline s_1,\dots,\underline s_{p-1}\right\}$ be an orthonormal basis of $\,\mathbb R^p$. Then it is clear that 
\[\partial_{\underline x}=\underline t\langle\underline t,\partial_{\underline x}\rangle+\sum_{j=1}^{p-1}\underline s_j\langle\underline s_j,\partial_{\underline x}\rangle.\]
Due to the fact that 
\begin{equation}\label{iguathet}
\langle\underline s_j,\partial_{\underline x}\rangle f(\theta,\rho)=\langle\underline s_j,\underline t\rangle\left(\partial_{\theta}f\right)=0
\end{equation}
we may conclude that any null solution of the operator $\langle\underline t,\partial_{\underline x}\rangle-\underline t\partial_{\underline y}$ of the form (\ref{axialwfunc}) is also a null solution of $\partial_{\underline X}=\partial_{\underline x}+\partial_{\underline y}$, i.e. it is monogenic.
\end{rem}

The following lemmas will be essential to obtain the main result in the section. The properties of the operators we list in the first lemma may be proved by induction and can be found in \cite{D}.

\begin{lem}\label{operatorsD}
Suppose that $h(x_1,x_2)$ is an infinitely differentiable $\mathbb R$-valued function. Then for $j=1,2$ one has 
\begin{itemize}
\item[{\rm i.}] $\partial_{x_j}^2\left(x_j^{-1}\partial_{x_j}\right)^nh=\left(x_j^{-1}\partial_{x_j}\right)^n\left(\partial_{x_j}^2h\right)-2n\left(x_j^{-1}\partial_{x_j}\right)^{n+1}h$, 
\item[{\rm ii.}] $\partial_{x_j}^2\left(\partial_{x_j}\,x_j^{-1}\right)^nh=\left(\partial_{x_j}\,x_j^{-1}\right)^n\left(\partial_{x_j}^2h\right)-2n\left(\partial_{x_j}\,x_j^{-1}\right)^{n+1}h$,
\item[{\rm iii.}] $\left(\partial_{x_j}\,x_j^{-1}\right)^n\left(\partial_{x_j}h\right)=\partial_{x_j}\left(x_j^{-1}\partial_{x_j}\right)^nh$,  
\item[{\rm iv.}] $\left(x_j^{-1}\partial_{x_j}\right)^n\left(\partial_{x_j}h\right)-\partial_{x_j}\left(\partial_{x_j}\,x_j^{-1}\right)^nh=\displaystyle{\frac{2n}{x_j}}\,\left(\partial_{x_j}\,x_j^{-1}\right)^nh$.
\end{itemize}
\end{lem}

\begin{lem}\label{indents-Delta}
Let $h$ be a harmonic $\mathbb R$-valued function of two variables defined on an open subset of the upper half-plane. Then 
\begin{align*}
\left(\langle\underline t,\partial_{\underline x}\rangle^2+\Delta_{\underline y}\right)^n\big(h(\theta,\rho)P_{\ell}(\underline y)\big)&=\prod_{j=1}^n\big(2\ell+q-(2j-1)\big)\left(\left(\rho^{-1}\partial_\rho\right)^nh\right)P_{\ell},\\
\left(\langle\underline t,\partial_{\underline x}\rangle^2+\Delta_{\underline y}\right)^n\big(h(\theta,\rho)\,\underline\nu P_{\ell}(\underline y)\big)&=\prod_{j=1}^n\big(2\ell+q-(2j-1)\big)\left(\left(\partial_\rho\,\rho^{-1}\right)^nh\right)\underline\nu P_{\ell},
\end{align*}
with $n$ a positive integer.
\end{lem}
\begin{proof}
We first prove that for any twice continuously differentiable $\mathbb R$-valued function $g$ of two variables defined on an open subset of the upper half-plane the following equalities hold  
\begin{equation}\label{eqdelta1}
\left(\langle\underline t,\partial_{\underline x}\rangle^2+\Delta_{\underline y}\right)\big(g(\theta,\rho)P_{\ell}(\underline y)\big)=\left(\partial_{\theta}^2g+\partial_{\rho}^2g+(2\ell+q-1)\frac{\partial_{\rho}g}{\rho}\right)P_{\ell},
\end{equation}
\begin{equation}\label{eqdelta2}
\left(\langle\underline t,\partial_{\underline x}\rangle^2+\Delta_{\underline y}\right)\big(g(\theta,\rho)\underline\nu P_{\ell}(\underline y)\big)=\left(\partial_{\theta}^2g+\partial_{\rho}^2g+(2\ell+q-1)\partial_\rho\left(\frac{g}{\rho}\right)\right)\underline\nu P_{\ell}.
\end{equation}
In fact, it  is easily seen that
\begin{align*}
\langle\underline t,\partial_{\underline x}\rangle^2g&=\partial_{\theta}^2g,\\
\Delta_{\underline y}g&=-\partial_{\underline y}^2g=-\partial_{\underline y}(\underline\nu\partial_{\rho}g)=\partial_{\rho}^2g+\frac{q-1}{\rho}\,\partial_{\rho}g.
\end{align*}
Therefore 
\begin{align*}
\Delta_{\underline y}(gP_{\ell})&=(\Delta_{\underline y} g)P_{\ell}+2\sum_{j=1}^q(\partial_{x_{p+j}}g)(\partial_{x_{p+j}}P_{\ell})+g(\Delta_{\underline y}P_{\ell})\\
&=\left(\partial_{\rho}^2g+\frac{q-1}{\rho}\,\partial_{\rho}g\right)P_{\ell}+2\frac{\partial_{\rho}g}{\rho}\sum_{j=1}^qx_{p+j}\partial_{x_{p+j}}P_{\ell}\\
&=\left(\partial_{\rho}^2g+\frac{2\ell+q-1}{\rho}\,\partial_{\rho}g\right)P_{\ell},
\end{align*}
where we have used Euler's theorem for homogeneous functions. 

In the same spirit we also obtain 
\begin{align*}
\Delta_{\underline y}(g\underline\nu P_{\ell})&=(\Delta_{\underline y}\,\underline\nu)gP_k+2\sum_{j=1}^q(\partial_{x_{p+j}}\underline\nu)(\partial_{x_{p+j}}(gP_{\ell}))+\underline\nu\Delta_{\underline y}(gP_{\ell})\\
&=-\frac{(q-1)}{\rho^2}\,g\underline\nu P_{\ell}+2\sum_{j=1}^q\left(\frac{e_{p+j}}{\rho}-\frac{x_{p+j}}{\rho^2}\,\underline\nu\right)\left(\frac{x_{p+j}}{\rho}\,(\partial_{\rho}g)P_{\ell}+g(\partial_{x_{p+j}}P_{\ell})\right)\\
&\qquad+\left(\partial_{\rho}^2g+\frac{2\ell+q-1}{\rho}\,\partial_{\rho}g\right)\underline\nu P_{\ell}\\
&=\left(\partial_{\rho}^2g+(2\ell+q-1)\left(\frac{\partial_{\rho}g}{\rho}-\frac{g}{\rho^2}\right)\right)\underline\nu P_\ell.
\end{align*}
The proof now follows by induction using equalities (\ref{eqdelta1}) and (\ref{eqdelta2}) together with statements i and ii of Lemma \ref{operatorsD}. It is clear that the lemma is true in the case $n=1$. Assume that the formulae hold for a positive integer $n$; we will prove them for $n+1$. 

We thus get
\begin{align*}
&\frac{\left(\langle\underline t,\partial_{\underline x}\rangle^2+\Delta_{\underline y}\right)^{n+1}\big(hP_{\ell}\big)}{\prod_{j=1}^n\big(2\ell+q-(2j-1)\big)}=\left(\langle\underline t,\partial_{\underline x}\rangle^2+\Delta_{\underline y}\right)\left(\left(\left(\rho^{-1}\partial_{\rho}\right)^nh\right)P_{\ell}\right)\\
&\qquad=\left(\partial_{\theta}^2\left(\rho^{-1}\partial_{\rho}\right)^nh+\partial_{\rho}^2\left(\rho^{-1}\partial_{\rho}\right)^nh+(2\ell+q-1)\left(\rho^{-1}\partial_{\rho}\right)^{n+1}h\right)P_{\ell}\\
&\qquad=\left(\left(\rho^{-1}\partial_{\rho}\right)^n\left(\partial_{\theta}^2h+\partial_{\rho}^2h\right)+(2\ell+q-(2n+1))\left(\rho^{-1}\partial_{\rho}\right)^{n+1}h\right)P_{\ell}\\
&\qquad=(2\ell+q-(2n+1))\left(\left(\rho^{-1}\partial_{\rho}\right)^{n+1}h\right)P_\ell,
\end{align*}
which establishes the first formula. The other one may be proved in a similar way.
\end{proof}

These lemmas lead to Fueter's theorem for embedded axial monogenics depending on a variable axis along $\underline t\in S^{p-1}$.

\begin{thm}\label{FueSomwave}
Suppose that  the function $u+iv$ is holomorphic in the open subset $\Xi$ of the upper half-plane. If $q$ is odd, then the function
\begin{equation*}
\left(\langle\underline t,\partial_{\underline x}\rangle^2+\Delta_{\underline y}\right)^{\ell+\frac{q-1}{2}}\Big(\big(u(\theta,\rho)-\underline t\,\underline\nu\,v(\theta,\rho)\big)P_\ell(\underline y)\Big),
\end{equation*}
is a null solution of the operator $\langle\underline t,\partial_{\underline x}\rangle-\underline t\partial_{\underline y}$ in $\Omega=\{(\underline x,\underline y)\in\mathbb R^{m}:\;(\theta,\rho)\in\Xi\}$.
\end{thm}
\begin{proof}
By Lemma \ref{indents-Delta}, we get that 
\begin{multline*}
\left(\langle\underline t,\partial_{\underline x}\rangle^2+\Delta_{\underline y}\right)^{\ell+\frac{q-1}{2}}\Big(\big(u(\theta,\rho)-\underline t\,\underline\nu\,v(\theta,\rho)\big)P_{\ell}(\underline y)\Big)\\
=(2\ell+q-1)!!\bigl(M(\theta,\rho)-\underline t\,\underline\nu N(\theta,\rho)\bigr)P_{\ell}(\underline y),
\end{multline*}
with 
\[M=\left(\rho^{-1}\partial_{\rho}\right)^{\ell+\frac{q-1}{2}}u,\]
\[N=\left(\partial_\rho\,\rho^{-1}\right)^{\ell+\frac{q-1}{2}}v.\]
On account of Lemma \ref{axialwavef} it is sufficient to show that $M$ and $N$ fulfill the Vekua-type system (\ref{VEopeFS}). Using statements iii and iv of Lemma \ref{operatorsD} and the fact that $u+iv$ is holomorphic we obtain
\[\begin{split}
\partial_{\theta}M-\partial_{\rho}N&=\left(\rho^{-1}\partial_{\rho}\right)^{\ell+\frac{q-1}{2}}\left(\partial_{\theta}u\right)-\partial_{\rho}\left(\partial_\rho\,\rho^{-1}\right)^{\ell+\frac{q-1}{2}}v\\
&=\left(\rho^{-1}\partial_{\rho}\right)^{\ell+\frac{q-1}{2}}\left(\partial_{\rho}v\right)-\partial_{\rho}\left(\partial_\rho\,\rho^{-1}\right)^{\ell+\frac{q-1}{2}}v\\
&=\frac{2\ell+q-1}{\rho}\left(\partial_\rho\,\rho^{-1}\right)^{\ell+\frac{q-1}{2}}v=\frac{2\ell+q-1}{\rho}N
\end{split}\]
and
\[\begin{split}
\partial_{\rho}M+\partial_{\theta}N&=\partial_{\rho}\left(\rho^{-1}\partial_{\rho}\right)^{\ell+\frac{q-1}{2}}u+\left(\partial_\rho\,\rho^{-1}\right)^{\ell+\frac{q-1}{2}}\left(\partial_{\theta}v\right)\\
&=\left(\partial_\rho\,\rho^{-1}\right)^{\ell+\frac{q-1}{2}}\left(\partial_{\rho}u\right)+\left(\partial_\rho\,\rho^{-1}\right)^{\ell+\frac{q-1}{2}}\left(\partial_{\theta}v\right)\\
&=\left(\partial_\rho\,\rho^{-1}\right)^{\ell+\frac{q-1}{2}}(\partial_{\rho}u+\partial_{\theta}v)=0,
\end{split}\]
which completes the proof.
\end{proof}

\section{Fueter-Funk-Hecke theorem}\label{sect3}

In this section we shall state and proof the main result of the paper. It will be a combination of Theorem \ref{FueSomwave} and the so-called Funk-Hecke's formula (see e.g. \cite{Hoch}).

\begin{thm}[Funk-Hecke's formula]
Suppose that $\int_{-1}^1\vert F(t)\vert(1-t^2)^{(p-3)/2}dt<\infty$ and let $\underline\xi\in S^{p-1}$. If $Y_k(\underline x)$ is a spherical harmonic of order $k$ in $\mathbb R^p$, then 
\[\int_{S^{p-1}}F(\langle\underline\xi,\underline\eta\rangle)Y_k(\underline\eta)dS(\underline\eta)=Y_{k}(\underline\xi)\,\vert S^{p-2}\vert\,C_k(1)^{-1}\int_{-1}^1F(t)C_k(t)(1-t^2)^{(p-3)/2}dt,\]
where $C_k(t)$ is the Gegenbauer polynomial $C^{\lambda}_k(t)$ with $\lambda=(p-2)/2$ and $\vert S^{p-2}\vert$ denotes the surface area of the unit sphere in $\mathbb R^{p-1}$.
\end{thm}

We recall that the  Gegenbauer polynomial $C^{\lambda}_k(t)$ are orthogonal polynomials on the interval $[-1,1]$ with respect to the weight function $(1-t^2)^{\lambda-1/2}$ and satisfy the recurrence relation
\begin{align*}
C^{\lambda}_k(t)&=\frac{1}{k}\left(2(k+\lambda-1)t\,C^{\lambda}_{k-1}(t)-(k+2\lambda-2)C^{\lambda}_{k-2}(t)\right),\quad k\ge2,\\
C^{\lambda}_0(t)&=1,\qquad C^{\lambda}_1(t)=2\lambda t.
\end{align*}
For any $\underline x\in\mathbb R^p$ and $\underline y\in\mathbb R^q$ we put $\underline\omega=\underline x/r$, $\underline\nu=\underline y/\rho$, with $r=\vert\underline x\vert$ and $\rho=\vert\underline y\vert$.  Let  $P_k(\underline x)$, $P_\ell(\underline y)$ be homogeneous monogenic polynomials. For simplicity we shall assume that $P_k(\underline x)$, $P_\ell(\underline y)$ takes values in the even subalgebra $\mathbb R_{0,p}^+$, $\mathbb R_{0,q}^+$ respectively and hence
\[P_k(\underline x)P_\ell(\underline y)=P_\ell(\underline y)P_k(\underline x).\]

\begin{thm}[Fueter-Funk-Hecke theorem]\label{FrankFF}
Let $h=u+iv$ be a holomorphic function in  an open subset $\Xi$ of the upper half-plane that is invariant under the dilations $(r,\rho)\rightarrow(rt,\rho)$, $t\in[-1,1]$. In case $q$ is odd, then the function
\[\mathsf{Ft}_{p,q}\left[h(z),P_k(\underline x),P_\ell(\underline y)\right](\underline x,\underline y)=(\Delta_{\underline x}+\Delta_{\underline y})^{\ell+\frac{q-1}{2}}\Big(\big(A(r,\rho)-\underline\omega\,\underline\nu B(r,\rho)\big)P_k(\underline x)P_\ell(\underline y)\Big)\]
is monogenic in $\Omega=\left\{\left(\underline x,\underline y\right)\in\mathbb R^{m}:\;(r,\rho)\in\Xi\right\}$, where
\begin{equation}\label{parteA}
A(r,\rho)=C_k(1)^{-1}r^{-k}\int_{-1}^1u(rt,\rho)C_k(t)(1-t^2)^{(p-3)/2}dt,
\end{equation}
\begin{equation}\label{parteB}
B(r,\rho)=C_{k+1}(1)^{-1}r^{-k}\int_{-1}^1v(rt,\rho)C_{k+1}(t)(1-t^2)^{(p-3)/2}dt.
\end{equation}
\end{thm}
\begin{proof}
Let $\left\{\underline t,\underline s_1,\dots,\underline s_{p-1}\right\}$ be an orthonormal basis of $\,\mathbb R^p$. Then it is clear that 
\[\partial_{\underline x}=\underline t\langle\underline t,\partial_{\underline x}\rangle+\sum_{j=1}^{p-1}\underline s_j\langle\underline s_j,\partial_{\underline x}\rangle,\]
\[\Delta_{\underline x}=\langle\underline t,\partial_{\underline x}\rangle^2+\sum_{j=1}^{p-1}\langle\underline s_j,\partial_{\underline x}\rangle^2.\]
Consider the function
\[I(\underline x,\underline y)=\left(\int_{S^{p-1}}\big(u(\theta,\rho)-\underline t\,\underline\nu\,v(\theta,\rho)\big)P_k(\underline t)dS(\underline t)\right)P_\ell(\underline y).\]
It thus follows from (\ref{iguathet}) that
\begin{multline*}
(\Delta_{\underline x}+\Delta_{\underline y})^{\ell+\frac{q-1}{2}}I(\underline x,\underline y)\\
=\int_{S^{p-1}}\left(\langle\underline t,\partial_{\underline x}\rangle^2+\Delta_{\underline y}\right)^{\ell+\frac{q-1}{2}}\Big(\big(u(\theta,\rho)-\underline t\,\underline\nu\,v(\theta,\rho)\big)P_\ell(\underline y)\Big)P_k(\underline t)dS(\underline t).
\end{multline*}
Therefore by Theorem \ref{FueSomwave} and Remark \ref{obser1} we can assert that $(\Delta_{\underline x}+\Delta_{\underline y})^{\ell+\frac{q-1}{2}}I(\underline x,\underline y)$ is monogenic, i.e.
\[(\partial_{\underline x}+\partial_{\underline y})(\Delta_{\underline x}+\Delta_{\underline y})^{\ell+\frac{q-1}{2}}I=0.\] 
We now write $I$ as 
\[I(\underline x,\underline y)=\left(I_1(\underline x,\underline y)+\underline\nu I_2(\underline x,\underline y)\right)P_\ell(\underline y)\]
with
\[I_1(\underline x,\underline y)=\int_{S^{p-1}}u(\theta,\rho)P_k(\underline t)dS(\underline t),\quad I_2(\underline x,\underline y)=\int_{S^{p-1}}\underline t\,v(\theta,\rho)P_k(\underline t)dS(\underline t).\]
As $P_k(\underline x)$, $\underline x P_k(\underline x)$ are both harmonic we can make use of Funk-Hecke's formula to compute $I_1$ and $I_2$. Indeed, writing $\underline x$ as $r\underline\omega$ we obtain
\[I_1(\underline x,\underline y)=P_k(\underline\omega)\vert S^{p-2}\vert\,C_k(1)^{-1}\int_{-1}^1u(rt,\rho)C_k(t)(1-t^2)^{(p-3)/2}dt,\]
\[I_2(\underline x,\underline y)=\underline\omega P_k(\underline\omega)\vert S^{p-2}\vert\,C_{k+1}(1)^{-1}\int_{-1}^1v(rt,\rho)C_{k+1}(t)(1-t^2)^{(p-3)/2}dt,\]
and thus completing the proof.
\end{proof}

\begin{rem}
Note that Theorem \ref{FrankFF} produces biaxial monogenic functions of the form
\[\bigl(M(r,\rho)+\underline\omega\,\underline\nu\,N(r,\rho)\bigr)P_k(\underline x)P_{\ell}(\underline y),\] 
where $M$, $N$ are $\mathbb R$-valued functions satisfying the Vekua-type system
\begin{equation*}
\begin{split}
\partial_{r}M+\partial_{\rho}N&=-\frac{2\ell+q-1}{\rho}N\\
\partial_{\rho}M-\partial_{r}N&=\frac{2k+p-1}{r}N.
\end{split}
\end{equation*}
\end{rem}
\noindent
Here we present three examples:
\begin{align*}
&\mathsf{Ft}_{3,3}\left[iz,1,1\right](\underline x,\underline y)=\frac{1}{\rho}-\frac{\underline x\,\underline y}{3\rho^3},\\
&\mathsf{Ft}_{4,3}\left[iz^4,P_1(\underline x),1\right](\underline x,\underline y)=\left(\frac{r^2-6\rho^2}{\rho}-\frac{\underline x\,\underline y}{8\rho^3}\big(r^2+8\rho^2\big)\right)P_1(\underline x),\\
&\mathsf{Ft}_{3,3}\left[1/(1+z^2),1,1\right](\underline x,\underline y)=\frac{4}{\big(r^2+(\rho+1)^2\big)\big(r^2+(\rho-1)^2\big)}\\
&-\underline\omega\,\underline\nu\left(\frac{2(r^2-\rho^2+1)}{r\rho\big(r^2+(\rho+1)^2\big)\big(r^2+(\rho-1)^2\big)}+\frac{\arctan\left(\displaystyle{\frac{r}{\rho+1}}\right)+\arctan\left(\displaystyle{\frac{r}{\rho-1}}\right)}{r^2\rho^2}\right).
\end{align*}
We end the paper with a brief discussion about the action of the Fueter-Funk-Hecke mapping $\mathsf{Ft}_{p,q}$ on the positive powers of $z$. We shall first consider the case $h(z)=z^{2n}$. For this case we clearly have that
\begin{align*}
u(rt,\rho)&=\sum_{\mu=0}^{n}(-1)^{n-\mu}\binom{2n}{2\mu}(rt)^{2\mu}\rho^{2n-2\mu},\\
v(rt,\rho)&=\sum_{\mu=0}^{n-1}(-1)^{n-\mu-1}\binom{2n}{2\mu+1}(rt)^{2\mu+1}\rho^{2n-2\mu-1},
\end{align*}
which implies that  $u(rt,\rho)C_k(t)$, $v(rt,\rho)C_{k+1}(t)$ are both even functions in the variable $t$ when $k$ is even and odd functions when $k$ is odd. 

Hence, we can assume that $k$ is even since for $k$ odd the functions $A(r,\rho)$, $B(r,\rho)$ defined respectively in (\ref{parteA}), (\ref{parteB}) are equal to zero. Taking into account the orthogonality of the polynomials $C_k(t)$, $C_{k+1}(t)$ we have 
\[\int_{-1}^1t^{2\mu}C_k(t)(1-t^2)^{(p-3)/2}dt=\int_{-1}^1t^{2\mu+1}C_{k+1}(t)(1-t^2)^{(p-3)/2}dt=0,\quad 0\le\mu<k/2.\]
It thus follows that
\[A(r,\rho)=\sum_{\mu=k/2}^{n}a_{\mu}r^{2\mu-k}\rho^{2n-2\mu},\qquad B(r,\rho)=\sum_{\mu=k/2}^{n-1}b_{\mu}r^{2\mu-k+1}\rho^{2n-2\mu-1}\]
for certain real constants $a_{\mu}$, $b_{\mu}$ and as a result $A(r,\rho)-\underline\omega\,\underline\nu B(r,\rho)$ will be a homogeneous polynomial of degree $2n-k$ with real coefficients in the variables $\underline x$, $\underline y$. A similar analysis can be made for the case of odd powers of $z$ and we arrive at the following conclusion.
\[\mathsf{Ft}_{p,q}\left[z^n,P_k(\underline x),P_{\ell}(\underline y)\right](\underline x,\underline y)=\left\{\begin{array}{ll}H(\underline x,\underline y)P_k(\underline x)P_{\ell}(\underline y),&k-n\;\;\text{even}\\0,&k-n\;\;\text{odd},\end{array}\right.\]
 where $H(\underline x,\underline y)$ denotes a homogeneous polynomial of degree $n-(k+2\ell+q-1)$ with real coefficients in the variables $\underline x$, $\underline y$ if $n\ge k+2\ell+q-1$ and zero otherwise. 
 
As an illustration, we provide the following examples:  

\begin{align*}
\mathsf{Ft}_{3,3}\left[z^4,1,1\right](\underline x,\underline y)&=\underline x^2+\frac{2}{3}\underline x\,\underline y-\underline y^2,\\
\mathsf{Ft}_{3,3}\left[z^7,P_1(\underline x),1\right](\underline x,\underline y)&=\left(3\underline x^4+4\underline x^3\underline y-14\underline x^2\underline y^2-\frac{28}{5}\underline x\,\underline y^3+7\underline y^4\right)P_1(\underline x).
\end{align*}

\subsection*{Acknowledgments}

D. Pe\~na Pe\~na acknowledges the support of a Postdoctoral Fellowship funded by the \lq\lq Special Research Fund" (BOF) of Ghent University.

\end{document}